\documentclass[10pt]{amsart}

\usepackage[T1]{fontenc}
\usepackage[utf8]{inputenc}
\usepackage{lmodern}

\usepackage{amsmath,amssymb,amsthm,mathtools}
\usepackage{mathrsfs}

\usepackage{microtype}
\usepackage{enumitem}
\usepackage{tikz-cd}

\usepackage{xcolor}
\usepackage[
colorlinks=true,
linkcolor=red,
citecolor=blue,
urlcolor=blue
]{hyperref}

\numberwithin{equation}{section}
\theoremstyle{plain}
\newtheorem{theorem}{Theorem}[section]
\newtheorem{lemma}[theorem]{Lemma}
\newtheorem{corollary}[theorem]{Corollary}
\newtheorem{proposition}[theorem]{Proposition}
\theoremstyle{definition}
\newtheorem{definition}[theorem]{Definition}
\theoremstyle{remark}
\newtheorem{remark}[theorem]{Remark}

\newcommand{\mb}{\mathfrak{b}}

\newcommand{\vphi}{\varphi}
\newcommand{\mE}{\mathcal{E}}

\newcommand{\ps}{\preceq_{\mathrm{S}}}
\newcommand{\isd}{\preceq_{\mathrm{ISD}}}
\newcommand{\pd}{\preceq_{\mathrm{D}}}

\newcommand{\dd}{\Delta_{\mathrm{D}}}

\DeclareMathOperator{\C}{C*}

\title[Choquet-Type State-Space Relations]
{Choquet-Type Relations and a State Space Level Amendment of Arveson's Hyperrigidity Conjecture}

\author{Hridoyananda Saikia}
\address{Department of Mathematics, University of Haifa, Haifa, Israel, 3498838}
\email{hridoyanandasaikia@gmail.com}

\date{\today}

\begin{document}
	\maketitle

	\begin{abstract}
		Davidson and Kennedy introduced the dilation order in connection with commutative $\C$-algebras and classical Choquet theory. Its noncommutative counterpart continues to detect the unique extension property of GNS representations. Motivated by the failure of Arveson's hyperrigidity conjecture and the amended theorem of Clouâtre and Thompson, we develop a state-space approach to this rigidity phenomenon. First, we introduce the strong dilation relation on the state space of a $\C$-algebra and characterize the unique tight extension property through maximality of states. Second, we define the integral subdivision relation and prove that maximality of all pure states in the dilation order implies maximality of every state in the integral subdivision relation. This provides a state space-level amendment of Arveson's hyperrigidity conjecture and yields an alternative proof of the Clouâtre-Thompson theorem. Finally, we show that the strong dilation and integral subdivision relations coincide in the commutative setting and they both agree with the abstract Choquet order associated with a function system.
	\end{abstract}

	\section{Introduction}\label{sec:introduction}
	
	Arveson's boundary representations were introduced as a noncommutative analogue of the Choquet boundary: points of a compact space are replaced by irreducible representations of a \(\C\)-algebra \cite{Arveson_2,Arveson_1, dritschel2005boundary, MR3430455}. If \(S\) is an operator system in a unital \(\C\)-algebra \(A\), meaning a unital self-adjoint linear subspace of \(A\), and if \(\C(S)=A\), then a representation \(\pi:A\to B(H)\) is said to have the unique extension property with respect to \(S\) if the only unital completely positive (UCP) map $ \Phi:A\to B(H)$ satisfying $\Phi_{|S}= \pi_{|S}$ is \(\Phi=\pi\). A boundary representation for \(S\) is an irreducible
	representation of \(A\) with this property. Boundary representations have since become one of the central tools for understanding how much of the ambient
	\(\C\)-algebra is remembered by the generating operator system.
	
	One of the main problems in this direction was Arveson's hyperrigidity conjecture \cite{Arveson_Boundary_2}. The conjecture predicted that if all irreducible representations of \(A=\C(S)\) are boundary representations for \(S\), then
	every representation of \(A\) should have the unique extension property with
	respect to \(S\). This problem led to a substantial body of work connecting
	operator systems, dilation theory, noncommutative Choquet theory, approximation
	properties, and boundary phenomena
	\cite{Kleski,essential_normality,UnperforatedPairs,NoncommutativePeaking,
		Hartz_raphael,Hyperrigidity_tensor,Hyperrigidity_correspondence,
		Full_CuntZ_DorOn,Guy_hyperrigid_subset,CrossedProdOp,hyperrigidityforfunctions,
		GelfandTrasformBoundary,boundaryprojection,pietrzycki2024hyperrigidity,
		hyperrigiditySpectrohedra,hyperrigidityConvex,RaphaelObstruct}. In the
	commutative case, Davidson and Kennedy gave a particularly useful formulation in
	terms of orders on spaces of measures, including the Choquet order, the dilation
	order, and the strong dilation order \cite{hyperrigidityforfunctions}.

	Davidson and Kennedy introduced the dilation order on the state space of a commutative $\C$-algebra in connection with function systems and classical Choquet theory \cite{hyperrigidityforfunctions}. This relation admits a natural extension to the noncommutative setting \cite{davidson2025noncommutative,saikia2025rigidity}. More precisely, if $S\subset A$ is an operator system with $\C(S)=A$, then the pair $(S,A)$ determines a partial order $\preceq_{\mathrm D}$ on the state space $ \mE(A)$, called the \emph{dilation order}. This order provides a useful state-space formulation of the unique extension property: a state $\varphi\in\mathcal E(A)$ is maximal for $\preceq_{\mathrm D}$ if and only if its GNS representation $\pi_\varphi$ has the unique extension property with respect to $S$. Under the GNS correspondence, pure states are precisely those states whose GNS representations are irreducible. Consequently, the hypothesis that every irreducible representation of $A$ is a boundary representation for $S$ is equivalent to the assertion that every pure state of $A$ is maximal in the dilation order. The state-theoretic reformulation of Arveson's hyperrigidity conjecture is therefore the following. \par
	\medskip
	\noindent
	\textbf{State theoretic reformulation of Arveson's hyperrigidity conjecture.} \emph{If every pure state is maximal in $\pd$, then every state is maximal in $\pd$.}
	\medskip
	
	The original conjecture is now known to be false. Bilich and Dor-On constructed a counterexample \cite{CounterAdamBilich}, Scherer later produced a finite-dimensional counterexample \cite{Counter_finite}, and Clou\^atre subsequently developed a new obstruction yielding further counterexamples \cite{RaphaelObstruct}. From the state space point of view, these examples show that the passage from pure state maximality to arbitrary-state maximality is too strong for the dilation order $\pd$. After the failure of the original conjecture, the natural problem is therefore to identify the correct form of rigidity encoded by boundary representations, and to understand it at the level of the state space.
	
	Clouâtre and Thompson proposed such a correction through the unique tight
	extension property \cite{clouatre2024rigidity}. Instead of allowing completely positive extensions to take values in \(B(H)\), one requires them to take values in the von Neumann algebra generated by the representation. Thus a representation \(\pi:A\to B(H)\) has the unique tight extension property with respect to \(S\) if the only unital completely positive map $\Phi:A\to \pi(A)''$ satisfying \(\Phi_{|S}=\pi_{|S}\)  is \(\Phi=\pi\). Clouâtre and
	Thompson proved the amended hyperrigidity theorem: all irreducible
	representations are boundary representations if and only if every representation has the unique tight extension property \cite[Theorem~2.4]{clouatre2024rigidity}. The idea of restricting the range of a UCP extension to a generated von Neumann algebra had already appeared in \cite{Kleski}.
	
	The first aim of this paper is to give a state-space explanation of this tight phenomenon. Motivated by the strong dilation order for measures studied by Davidson and Kennedy \cite{hyperrigidityforfunctions}, we define the \emph{strong dilation relation} \(\preceq_{\mathrm S}\) on \(\mE(A)\). Our first result identifies this relation as the state-space shadow of the unique tight extension property. This is established in Theorem \ref{T:UTEPmaximality}, our first main result of this paper. \par
	
	\medskip
	\noindent
	\textbf{Theorem A.}
	\emph{Let $S$ be an operator system in a unital separable $\C$-algebra $A$ such that $\C(S)=A$ and let $\vphi \in \mE(A)$ with the GNS representation $(\pi_{\vphi}, H_{\vphi}, \xi_{\vphi})$. Then $\pi_{\vphi}$ has the unique tight extension property if and only if $\vphi$ is maximal in the strong dilation relation.} 
	
	\medskip
	
	In view of Theorem~A, the Clou\^atre--Thompson amended hyperrigidity theorem admits the following equivalent state-space formulation: if every pure state of $A$ is maximal in the dilation order, then every state of $A$ is maximal in the strong dilation relation. We prove this implication as a consequence of a stronger maximality principle. To this end, we introduce the \emph{integral subdivision relation}, a new relation on $\mathcal E(A)$ determined by the pair $(S,A)$. We then show that it lies between the strong dilation relation and the ordinary dilation order in Theorem \ref{T:implication}:\par
	\medskip
	\noindent
	\textbf{Theorem B.} $\,\,\,\,\,\varphi\preceq_{\mathrm S}\psi
	\quad\Longrightarrow\quad
	\varphi\preceq_{\mathrm{ISD}}\psi
	\quad\Longrightarrow\quad
	\varphi\preceq_{\mathrm D}\psi.$ \par
	\medskip
	
	The definition of $\preceq_{\mathrm{ISD}}$ is designed to interact with
	barycentric decompositions of states. This allows us to establish the
	following maximality principle: if every pure state is maximal in the
	dilation order, then every state is maximal in the integral subdivision
	relation; see Proposition \ref{P:maximalitypassage}. Since maximality of a state for $\isd$ implies maximality in $\ps$, Proposition \ref{P:maximalitypassage} strengthens the state space formulation of the
	Clouâtre--Thompson theorem. Together with Theorem~A, it also yields an
	alternative proof of their amended hyperrigidity theorem (see Corollary \ref{C:CTamended}).

	The third aim of the paper is to show that the integral subdivision relation is the correct noncommutative analogue of the classical Choquet order. Let \(X\) be a compact metrizable space and let \(S\subset C(X)\) be a function system which separates the points of \(X\). The function system determines a cone \(\mathcal C_S\subset C(X)_{\mathrm{sa}}\), namely the uniformly closed max-stable cone generated by \(S_{\mathrm{sa}}\). This cone is equivalently the collection of pullbacks of continuous convex functions on the state space \(\mE(S)\) of $S$, along the evaluation map $\iota:X\to \mE(S)$.
	It therefore defines an \emph{abstract} Choquet order on Borel probability measures on \(X\) \cite[I.5]{alfsen2012compact}:
	\[
	\mu\preceq_{\mathrm C}\nu \quad\Longleftrightarrow\quad \int_X f\,d\mu\leq \int_X f\, d\nu \quad\text{for all } f\in \mathcal C_S.
	\]
	Combining the Choquet--Cartier disintegration theorem with our state-space
	relations gives the following characterization.
	
	\medskip
	\noindent
	\textbf{Theorem C.}
	\emph{Let \(X\) be a compact metrizable space, and let \(S\subset C(X)\) be a
		function system which separates the points of \(X\). For Borel probability
		measures \(\mu,\nu\) on \(X\), the following are equivalent:
		\[
		\mu\preceq_{\mathrm C}\nu,\qquad
		\mu\preceq_{\mathrm S}\nu,\qquad
		\mu\preceq_{\mathrm{ISD}}\nu.
		\]}
	
	\medskip

	The paper is organized as follows. Section~\ref{sec:preliminaries} collects the background material used throughout the paper. In Section~\ref{S:SD}, we introduce the strong dilation relation and relate maximality for this relation to the unique tight extension property. Section~\ref{S:ISD} develops the integral subdivision relation and gives a state-space reformulation of the amended hyperrigidity theorem. Subsection~\ref{SS:multiplicityfree} proves that the strong dilation and integral subdivision relations coincide in the multiplicity-free setting, while Subsection~\ref{SS:Choquet} identifies these relations with the abstract Choquet order through several equivalent characterizations.

	\section{Preliminaries}\label{sec:preliminaries}
	Throughout the paper, all $\C$-algebras are assumed to be unital, and they
	are assumed separable unless explicitly stated otherwise.  If $A$ is a $\C$-algebra, we denote its state space by $\mE(A)$, equipped with the weak-$*$ topology.

	Throughout, by a representation we mean a unital $*$-representation. Let $\vphi\in\mE(A)$. A \emph{representation of the state} $\vphi$ is a triple $(\pi,H,\xi)$, where $\pi:A\to B(H)$ is a representation and $\xi\in H$ is a unit vector such that
	\[
	\vphi(a)=\langle\pi(a)\xi,\xi\rangle,
	\qquad a\in A.
	\]
	The GNS representation of $\vphi$ will be denoted by
	$(\pi_{\vphi},H_{\vphi},\xi_{\vphi})$. It is a cyclic representation of $\vphi$, and $H_{\vphi} = \overline{\pi_{\vphi}(A)\xi_{\vphi}}$.

	For the general theory of $C^*$-algebras, we follow \cite{C*algebraswith,murphy}. For background on operator systems, unital completely positive maps, and related topics, we refer to \cite{paulsenbook}.
	
	\subsection{Dilation order and the unique extension property} 
	
	Let \(S\subset A\) be an operator system with $\C(S)=A$. A representation
	$\pi:A\to B(H)$ has the \emph{unique extension property}, or UEP, with respect to $S$ if the only unital completely positive (UCP) map $ \Phi:A\to B(H)$
	satisfying $\Phi_{|S}= \pi_{|S}$ is $\Phi=\pi$. A \emph{boundary representation} for $S$ is an irreducible
	representation of $A$ with the unique extension property.
	
	The tight version of this notion was introduced by Clouâtre and Thompson
	\cite{clouatre2024rigidity}. A representation $\pi:A\to B(H)$ has the \emph{unique tight extension property}, or UTEP, with respect to $S$ if the only unital completely positive map $ \Phi:A\to \pi(A)''$ satisfying $\Phi_{|S}= \pi_{|S}$ is  $\Phi=\pi$.
	
	Before defining the dilation order, we record the following result. It was established for commutative $\C$-algebras \cite[Section 6]{hyperrigidityforfunctions}, and the same argument applies in the noncommutative setting. For completeness, we include a brief proof sketch.
	
	\begin{proposition}
		Let $S$ be an operator system in a  $\C$-algebra $A$, such that $\C(S)=A$. Let $\vphi, \psi \in \mE(A)$ and assume that $(\pi_{\vphi},H_{\vphi},\xi_{\vphi})$ is the GNS representation of $\vphi$. Then the following are equivalent. 
		\begin{enumerate}
			\item There exists a UCP map $\Phi: A \to B(H_{\vphi})$ such that $\pi_{\vphi}(s)= \Phi(s)$ for all $s \in S$ and $\psi(a)= \langle \Phi(a)\xi_{\vphi}, \xi_{\vphi}\rangle $ for all $a \in A$. 
			\item For every representation $(\pi, H, \xi)$ of $\vphi$, there exists a UCP map $\Phi: A \to B(H)$ such that $\pi(s)= \Phi(s)$ for all $s \in S$ and $\psi(a)= \langle \Phi(a)\xi, \xi\rangle $ for all $a \in A$.
			\item There exists a representation $(\pi, H,\xi)$ of $\vphi$ and a UCP map $\Phi: A \to B(H)$ such that  $\pi(s)= \Phi(s)$ for all $s \in S$ and $\psi(a)= \langle \Phi(a)\xi, \xi\rangle $ for all $a \in A$.
		\end{enumerate}
	\end{proposition}
	
	\begin{proof}\label{P:equivalence}
		Every representation $(\pi,H,\xi)$ of $\vphi$ is, up to unitary equivalence, of the form
		\[
		H=H_{\vphi}\oplus H_0,\qquad \pi=\pi_{\vphi}\oplus\rho,\qquad \xi=\xi_{\vphi}\oplus 0
		\]
		for some representation $\rho:A\to B(H_0)$. Indeed, the cyclic subspace $\overline{\pi(A)\xi}$ is reducing for $\pi$ and is unitarily equivalent to the GNS representation of $\vphi$.
		
		Suppose that $(1)$ holds, and let $\Phi_0:A\to B(H_{\vphi})$ be the corresponding UCP map. For a representation of $\vphi$ decomposed as above, the UCP map $\Phi=\Phi_0 \oplus \rho$ satisfies the conditions in $(2)$. Thus $(1)\Rightarrow(2)$, while $(2)\Rightarrow(3)$ is immediate.
		
		Finally, suppose that $(3)$ holds. Using the decomposition above and compressing the corresponding UCP map $\Phi:A\to B(H)$ to the summand $H_{\vphi}$ gives a UCP map $\Phi_0: A \to B(H_{\vphi})$ by
		\[
		\Phi_0(a)=P_{H_{\vphi}}\Phi(a)|_{H_{\vphi}},
		\qquad a\in A.
		\]
		Since $\Phi(s)=\pi(s)$ for $s\in S$ and $\xi=\xi_{\vphi}\oplus0$, this map satisfies $\Phi_0(s)=\pi_{\vphi}(s)$ for every $s\in S$, and $\psi(a)= \langle\Phi_0(a)\xi_{\vphi},\xi_{\vphi}\rangle$ for every $a\in A$. Hence $(3)\Rightarrow(1)$.
	\end{proof}

	\begin{definition}
		Let $\varphi,\psi\in\mE(A)$. We write $\varphi\preceq_{\mathrm D}\psi$ if any, and hence all, of the equivalent conditions in Proposition~\ref{P:equivalence} hold.
	\end{definition}
	
	This is a partial order due to \cite[Proposition 7.2.8 and Theorem 8.5.1]{davidson2025noncommutative}. We refer to this order as the \emph{dilation order} associated with the pair $(S,A)$.  A state $\vphi \in \mE(A)$ is \emph{maximal} in the dilation order if for $\psi \in \mE(A)$,
	\[
	\varphi \pd\psi
	\quad\Longrightarrow\quad
	\psi=\varphi.
	\]
	We will denote the set of all maximal states in the dilation order by $\dd$. Then $\dd$ is a Borel measurable subset of $\mE(A)$ by \cite[Lemma 2.1 and Theorem 3.8]{boundaryprojection}. 
	
	The next result states that maximality in the dilation order precisely detects cyclic representations having the unique extension property. It was proved in the commutative setting by Davidson and Kennedy \cite[Theorem~7.6]{hyperrigidityforfunctions}; see also \cite[Theorem~3.2.1]{saikia2025rigidity} for the noncommutative version. For completeness, we include below a brief sketch of the proof in the noncommutative setting, which is a direct adaptation of the commutative argument.

	\begin{theorem}\label{T:dilationUEP}
		Let $S\subset A$ be an operator system with $\C(S)=A$, and let
		$\varphi\in \mE(A)$. Then $\varphi$ is maximal for the dilation order
		$\preceq_{\mathrm D}$ if and only if the GNS representation
		$\pi_\varphi$ has the unique extension property with respect to $S$.
	\end{theorem}
	\begin{proof}
		Suppose that $\pi_{\vphi}: A \to B(H_{\vphi})$ has the unique extension property relative to $S$ and let $\psi \in \mE(A)$ such that $\vphi \pd \psi$. Then there exists a UCP map $\Phi: A \to B(H_{\vphi})$ such that $\pi_{\vphi}(s)=\Phi(s)$ for all $s \in S$ and $\psi(a)= \langle \Phi(a)\xi_{\vphi}, \xi_{\vphi} \rangle$ for all $a \in A$. The unique extension property of $\pi_{\vphi}$ forces $\pi_{\vphi}= \Phi$. This implies that $\psi=\vphi$ and hence $\vphi$ is maximal in $\pd$.
		
		Conversely, suppose that $\vphi$ is maximal in the dilation order. Choose a maximal dilation $\sigma:A\to B(K)$ of $\pi_{\vphi}$, so that $\sigma$ has the unique extension property \cite[Proposition 2.4 and Theorem 2.5]{Arveson_bdd_1} and there is an isometry $V:H_{\vphi}\to K$ satisfying
		\[
		\pi_{\vphi}(s)=V^*\sigma(s)V,
		\qquad s\in S.
		\]
		Define $\psi(a)=\langle \sigma(a)V\xi_{\vphi},V\xi_{\vphi}\rangle$ for all $a \in A$. By Proposition~\ref{P:equivalence}, $\vphi\pd\psi$. Hence the maximality of $\vphi$ gives $\psi=\vphi$. Therefore, the cyclic subrepresentation of $\sigma$ generated by $V\xi_{\vphi}$ is unitarily equivalent to $\pi_{\vphi}$. Since this cyclic subrepresentation is a direct summand of $\sigma$, it has the unique extension property. Consequently, $\pi_{\vphi}$ has the unique extension property.
	\end{proof}

	More generally, if $\preceq$ is a relation on $\mE(A)$, a state
	$\varphi\in \mE(A)$ is called \emph{maximal} for $\preceq$ if
	\[
	\varphi\preceq\psi
	\quad\Longrightarrow\quad
	\psi=\varphi.
	\]
	
	\subsection{A lifting property}\label{SS:Lifting}
	Next, we recall lifting theorem of Maharam \cite{Maharam, Topic_in_Theory_of_Lifting}, which will be used in the proofs of several results in this article.
	
	Let $(X,\mathcal A,\mu)$ be a complete probability space, and let
	$\mathscr L^\infty(X,\mathcal A)$ denote the space of bounded
	$\mathcal A$-measurable functions on $X$. If $q: \mathscr L^\infty(X,\mathcal A) \to L^\infty(X,\mu)$ denotes the canonical quotient map, then the lifting theorem
	\cite[Definition 10.5.1 and Theorem 10.5.4]{Bogachev_Measure} yields a
	unital $*$-homomorphism $\ell :  L^\infty(X,\mu) \to \mathscr L^\infty(X,\mathcal A)$ such that $q\circ \ell=\operatorname{id}_{L^\infty(X,\mu)}$. Thus, $\ell ([f])$ is a bounded measurable representative of the
	equivalence class $[f]$.
	
	More generally, suppose that $\mu$ is a possibly incomplete probability
	measure on $(X,\mathcal A)$. Let $\mathcal A_\mu$ denote the completion
	of $\mathcal A$ with respect to $\mu$, and let $\overline{\mu}$ be the
	corresponding completed measure. The map $[f]_\mu\longmapsto [f]_{\overline{\mu}}$ defines a canonical isometric $*$-isomorphism
	\[
	L^\infty(X,\mu)\cong L^\infty(X,\overline{\mu}).
	\]
	Applying the lifting theorem to the complete probability space
	$(X,\mathcal A_\mu,\overline{\mu})$, we obtain a unital $*$-homomorphism $\ell\colon L^\infty(X,\mu)
	\longrightarrow \mathscr L^\infty(X,\mathcal A_\mu)$ such that $\ell([f])=f$ $\overline{\mu}$-almost everywhere for every bounded $\mathcal A$-measurable representative $f$ of $[f]$.
	
	\subsection{Commutative $\C$-algebras and classical Choquet order}
	
	Here we briefly recall the classical Choquet order, which is a partial
	order on the Borel probability measures on a compact convex set; see
	\cite{alfsen2012compact,phelps2002lectures}. Let $X$ be a compact
	Hausdorff space. By the Riesz--Markov--Kakutani theorem, the state space
	of $C(X)$ is naturally identified with the set $\operatorname{Prob}(X)$
	of regular Borel probability measures on $X$. Under this identification,
	\[
	\mu(f)=\int_X f\,d\mu, \qquad f\in C(X).
	\]
	The GNS representation associated with $\mu\in\operatorname{Prob}(X)$ is
	given by
	\[
	\mathcal H_\mu=L^2(X,\mu), \qquad \xi_\mu=\mathbf 1_X, \qquad
	\pi_\mu(f)=M_f,
	\]
	where $M_f$ denotes the multiplication operator by $f$. Identifying
	$L^\infty(X,\mu)$ with its multiplication representation on $L^2(X,\mu)$, we have
	\[
	\pi_\mu(C(X))'' = \{M_g:g\in L^\infty(X,\mu)\} \cong L^\infty(X,\mu).
	\]
	
	Let $K$ be a compact metrizable convex subset of a locally convex space,
	and let $A(K)$ denote the function system of continuous affine functions
	on $K$. Since $A(K)$ separates the points of $K$, the Stone--Weierstrass theorem gives $C^*(A(K))=C(K)$.
	
	\begin{definition}
		Let $\mu,\nu\in\operatorname{Prob}(K)$. We write $\mu\preceq_{\mathrm C}\nu$ if $\mu(f) \leq \nu(f)$ for every real-valued continuous convex function $f$ on $K$. This relation is called the \emph{Choquet order}.
	\end{definition}
	
	Following Davidson and Kennedy \cite[Definition 3.1]{hyperrigidityforfunctions}, we say that $\mu$ is dominated by $\nu$
	in the \emph{strong dilation order}, and write $\mu\ps\nu$,
	if there is a unital positive map $\Phi\colon C(K)\longrightarrow L^\infty(K,\mu)$ such that $\Phi(\mathfrak a)=\mathfrak a$ for all $\mathfrak a\in A(K)$ and $\displaystyle \int_K f\,d\nu= \int_K\Phi(f)\,d\mu$ for all $f \in C(K)$. 
	
	For $x\in K$, let
	\[
	\operatorname{Prob}_x(K) = \left\{
	\lambda\in\operatorname{Prob}(K): \int_K \mathfrak a\,d\lambda=\mathfrak a(x) \text{ for every } \mathfrak a\in A(K)
	\right\}.
	\]
	The following combines the equivalence of the Choquet and strong
	dilation orders with Cartier's theorem \cite[Definition 3.1 and Theorems 3.7 and 4.1]{hyperrigidityforfunctions}
	.
	
	\begin{theorem} \label{T:Choquet}
		Let $K$ be a compact metrizable convex set and let
		$\mu,\nu\in\operatorname{Prob}(K)$. The following statements are
		equivalent:
		\begin{enumerate}
			\item $\mu\preceq_{\mathrm C}\nu$;
			\item $\mu \ps \nu$;
			\item there is a family $\{\lambda_x\}_{x \in K}\subset \operatorname{Prob}(K)$
			such that $x\mapsto\lambda_x(f)$ is $\mu$-measurable for every
			$f\in C(K)$, $\lambda_x\in\operatorname{Prob}_x(K)$ for $\mu$-almost every $x\in K$, and
			\[
			\int_K f\,d\nu = \int_K\lambda_x(f)\,d\mu(x), \qquad f\in C(K).
			\]
		\end{enumerate}
	\end{theorem}

	\section{The strong dilation relation and the unique tight extension property}\label{S:SD}
	In this section we introduce the strong dilation relation associated with an operator system $S\subset A$, and we explain its connection with the unique tight extension property for representations. Strong dilation is generally more restrictive than ordinary dilation. The main result of this section shows that this stronger state-space relation exactly detects tight uniqueness: a cyclic representation has the unique tight extension property precisely when its corresponding GNS state is maximal in the strong dilation relation.
	
	Recall that for every state $\vphi \in \mE(A)$, the GNS representation of $\vphi$ is denoted by $(\pi_{\vphi}, H_{\vphi}, \xi_{\vphi})$. Next we define the strong dilation relation on $\mE(A)$, which is motivated by the strong dilation order introduced by Davidson and Kennedy for measures on a compact convex set $K$ \cite[Definition~3.1]{hyperrigidityforfunctions}.
	\begin{definition}[Strong dilation relation]
		Let $S$ be an operator system in a C*-algebra $A$ such that $\C(S)=A$, and let  $\vphi, \psi \in \mE(A)$. We define the \emph{strong dilation relation} $\ps$ on $\mE(A)$ by declaring that $\varphi\ps\psi$ if there exists a UCP map $\Phi: A \to \pi_{\vphi}(A)''$ such that
		\[
		\Phi(s)= \pi_{\vphi}(s), \qquad  s \in S,
		\]
		and 
		\[
		\psi(a)= \langle \Phi(a)\xi_{\vphi}, \xi_{\vphi} \rangle, \qquad  a \in A.
		\]
		Note that the strong dilation relation depends on both $S$ and $A$, and we call it the strong dilation relation associated with the pair $(S,A)$. However, the pair $(S,A)$ will be clearly understood throughout the paper from the context, and we will simply write $\ps$ without mentioning the pair $(S,A)$. Finally, $\vphi \in \mE(A)$ is said to be \emph{maximal} in $\ps$ if
		\[
		\vphi \ps \psi \implies \psi = \vphi.
		\]
	\end{definition}

	In the commutative setting, the strong dilation relation is a partial
	order, since it coincides with the Choquet order \cite[Theorem 3.7]{hyperrigidityforfunctions}. It is not known whether $\ps$ is transitive in the general noncommutative setting, and therefore we do not claim that it defines a partial order in this setting. Nevertheless, $\ps$ is reflexive and antisymmetric. Reflexivity follows by taking $\Phi=\pi_\varphi$.
	Moreover, strong dilation implies domination in the dilation order:
	\[
	\varphi\ps \psi \quad\Longrightarrow\quad
	\varphi\pd\psi.
	\]
	Consequently, if both $\varphi\ps\psi$ and
	$\psi\ps\varphi$, then
	\[
	\vphi\pd\psi
	\quad\text{and}\quad
	\psi\pd\vphi.
	\]
	Since $\pd$ is antisymmetric, it follows that
	$\vphi=\psi$.

	The next part of this section explores the connection between the strong dilation relation and the unique tight extension property. The next result shows that the unique tight extension property respects direct sums. 
	
	\begin{proposition}\label{P:UTEPdirectsum}
		Let $S$ be an operator system in a $\C$-algebra $A$ such that $\C(S)=A$. Let $I$ be an indexing set and let $\pi_i: A \longrightarrow B(H_i)$ be a representation of $A$ on $H_i$ for all $i \in I$.  Suppose that
		\[
		H=\bigoplus_{i\in I}H_i,
		\qquad
		\pi=\bigoplus_{i\in I}\pi_i.
		\]
		Then $\pi$ has the unique tight extension property with respect to $S$ if and only if each $\pi_i$ has the unique tight extension property with respect to $S$.
	\end{proposition}
	
	\begin{proof}
		Let $P_i$ denote the projection of $H$ onto $H_i$. Since $H_i$ reduces
		$\pi(A)$, we have $P_i\in \pi(A)'$. Hence every element of $\pi(A)''$ commutes with $P_i$, and compression gives a normal surjective $*$-homomorphism
		\[
		\Omega_i:\pi(A)''\to \pi_i(A)'',
		\qquad
		\Omega_i(x)=x|_{H_i}.
		\]
		Indeed, $\Omega_i(\pi(a))=\pi_i(a)$ for every $a\in A$. Moreover, the range of $\Omega_i$ is a weak-$*$ closed $*$-subalgebra of $B(H_i)$ containing
		$\pi_i(A)$, and it is contained in $\pi_i(A)''$; hence it equals
		$\pi_i(A)''$.
		
		Assume first that $\pi$ has the unique tight extension property. Since
		$\Omega_i$ is a weak-$*$ continuous surjective $*$-homomorphism and $\Omega_i\circ \pi=\pi_i$, it follows from \cite[Lemma~2.6]{clouatre2024rigidity} that $\pi_i$ has the unique
		tight extension property for every $i\in I$.
		
		Conversely, assume that each $\pi_i$ has the unique tight extension property. Let $\Phi:A\to \pi(A)''$ be a unital completely positive map such that $\Phi(s)=\pi(s)$ for all $s \in S$. For each $i\in I$, the map $\Omega_i\circ\Phi:A\to \pi_i(A)''$ is unital completely positive and satisfies $(\Omega_i\circ\Phi)(s)=\pi_i(s)$ for all $s \in S$. By the unique tight extension property of $\pi_i$, we obtain $\Omega_i(\Phi(a))=\pi_i(a)$ for all $a \in A$. Fix $a\in A$  and let
		\[
		x=\Phi(a)-\pi(a)\in \pi(A)''.
		\]
		Then
		\[
		x|_{H_i}=0,
		\qquad i\in I.
		\]
		Since $x$ commutes with each $P_i$, each subspace $H_i$ reduces $x$. Then $x=0$, since $H_i$'s spans $H$. Thus $\Phi(a)=\pi(a)$ for all $a \in A$. Since $a$ was arbitrary, $\Phi=\pi$. Hence $\pi$ has the unique tight extension property.
	\end{proof}
	
	Every representation of a $\C$-algebra $A$ decomposes as a direct sum of cyclic representations. Consequently, Proposition~\ref{P:UTEPdirectsum} reduces the verification of the unique tight extension property to the cyclic summands. Moreover, every cyclic representation of $A$ is unitarily equivalent to the GNS representation of an associated state. We next show that maximality in the strong dilation relation detects precisely those states whose GNS representations have the unique tight extension property. In the spirit of Theorem~\ref{T:dilationUEP}, the next result shows that the strong dilation relation is appropriate for detecting the unique tight extension property at the level of the state space.

	\begin{theorem}\label{T:UTEPmaximality}
		Let $S$ be an operator system in the $\C$-algebra $A$ such that $\C(S)=A$ and let $\vphi \in \mE(A)$ with the GNS representation $(\pi_{\vphi}, H_{\vphi}, \xi_{\vphi})$. Then $\pi_{\vphi}$ has the unique tight extension property if and only if $\vphi$ is maximal in the strong dilation relation. 
	\end{theorem}
	
	\begin{proof}
		Assume that $\pi_{\vphi}$ has the unique tight extension property. Let $\psi \in \mE(A)$ such that $\vphi \ps \psi$. Then by definition, there exists a UCP map $\Phi: A \to \pi_{\vphi}(A)''$ such that $\Phi(s)= \pi_{\vphi}(s)$ for all $s \in S$ and 
		\[
		\psi(a)= \langle \Phi(a)\xi_{\vphi}, \xi_{\vphi} \rangle, \qquad a \in A.
		\]
		Then $\Phi= \pi_{\vphi}$ because of the unique tight extension property of $\pi_{\vphi}$. Hence for all $a \in A$,
		\[
		\psi(a)= \langle \Phi(a)\xi_{\vphi}, \xi_{\vphi} \rangle= \langle \pi_{\vphi}(a)\xi_{\vphi}, \xi_{\vphi} \rangle = \vphi(a).
		\]
		Hence $\vphi$ is maximal in the strong dilation relation. \par 
		Conversely, assume that $\vphi$ is maximal in the strong dilation relation and let $\Phi: A \to \pi_{\vphi}(A)''$ be a UCP map with $\Phi(s)= \pi_{\vphi}(s)$ for all $s \in S$. For all $a \in A$, define $\psi(a)=\langle \Phi(a)\xi_{\vphi},\xi_{\vphi}\rangle$. Then $\psi \in \mE(A)$ and $\vphi \ps \psi$. By maximality of $\vphi$, we have $\psi= \vphi$. Let $(\rho,K,V)$ be the minimal Stinespring representation of $\Phi$, so that
		\[
		\Phi(a)=V^*\rho(a)V,\qquad a\in A.
		\]
		Since $\psi=\varphi$, we have
		\[
		\langle \rho(a)V\xi_{\vphi},\rho(b)V\xi_{\vphi}\rangle =  \langle \pi_{\vphi}(a)\xi_{\vphi},\pi_{\vphi}(b)\xi_{\vphi}\rangle,\qquad a,b\in A.
		\]
		Hence $W : H_{\vphi} \to K$, defined by the formula $W\pi_{\vphi}(a)\xi_{\vphi}= \rho(a)V\xi_{\vphi}$ on the dense subspace $\pi_{\vphi}(A)\xi_{\vphi}$, is a well-defined isometry. Moreover, $W\pi_{\vphi}(a)=\rho(a)W$ for every $a\in A$. \par 
		We claim that $V=W$. Let $\mathcal N=\{a\in A:V\pi_{\vphi}(a)\xi_{\vphi}=W\pi_{\vphi}(a)\xi_{\vphi}\}$. Then $\mathcal N$ is a norm-closed linear subspace of $A$ containing $1$. If $s\in S$, then
		\[
		V^*\rho(s)V=\Phi(s)=\pi_{\vphi}(s).
		\]
		Thus, with $P=VV^*$,
		\[
		P W\pi_{\vphi}(s)\xi_{\vphi}
		=P\rho(s)V\xi_{\vphi}
		=VV^*\rho(s)V\xi_{\vphi}
		=V\pi_{\vphi}(s)\xi_{\vphi}.
		\]
		Since both $W$ and $V$ are isometries, we have
		\[
		\|W\pi_{\vphi}(s)\xi_{\vphi}\|=\|\pi_{\vphi}(s)\xi_{\vphi}\|=\|V\pi_{\vphi}(s)\xi_{\vphi}\|,
		\]
		so $W\pi_{\vphi}(s)\xi_{\vphi}=V\pi_{\vphi}(s)\xi_{\vphi}$. Hence $S\subset \mathcal N$. Now let $a\in\mathcal N$ and $s\in S$. Then
		\[
		W\pi_{\vphi}(sa)\xi_{\vphi}=\rho(s)W\pi_{\vphi}(a)\xi_{\vphi}=\rho(s)V\pi_{\vphi}(a)\xi_{\vphi}.
		\]
		Again,
		\[
		P W\pi_{\vphi}(sa)\xi_{\vphi}
		=VV^*\rho(s)V\pi_{\vphi}(a)\xi_{\vphi}
		=V\pi_{\vphi}(s)\pi_{\vphi}(a)\xi_{\vphi}
		=V\pi_{\vphi}(sa)\xi_{\vphi}.
		\]
		Since $V$ and $W$ are isometries,
		\[
		\|W\pi_{\vphi}(sa)\xi_{\vphi}\|=\|\pi_{\vphi}(sa)\xi_{\vphi}\|=\|V\pi_{\vphi}(sa)\xi_{\vphi}\|,
		\]
		and therefore $W\pi_{\vphi}(sa)\xi_{\vphi}=V\pi_{\vphi}(sa)\xi_{\vphi}$ and hence $sa\in\mathcal N$. Therefore $\mathcal N$ contains the unital algebra generated by $S$.
		Since $\C(S)=A$, we obtain $\mathcal N=A$. Hence $V=W$ on the dense
		subspace $\pi_{\vphi}(A)\xi_{\vphi}$, and consequently $V=W$ on $H_{\vphi}$. It follows that
		\[
		\Phi(a)=V^*\rho(a)V=W^*\rho(a)W=W^*W\pi_{\vphi}(a)=\pi_{\vphi}(a),
		\]
		for all $a \in A$. Thus $\Phi=\pi_{\vphi}$, and $\pi_{\vphi}$ has the unique tight extension property.
	\end{proof}
	
	The following result shows that, for pure states, maximality with respect to the dilation order is equivalent to maximality with respect to the strong dilation relation.
	
	\begin{proposition}
		Let $S$ be an operator system in the unital $\C$-algebra $A$ such that $\C(S)=A$, and let $\vphi \in \mE(A)$ be a pure state with GNS representation $(\pi_{\vphi}, H_{\vphi}, \xi_{\vphi})$. Then the following are equivalent.
		\begin{enumerate}
			\item $\vphi$ is maximal in $\pd$.
			\item $\vphi$ is maximal in $\ps$.
			\item $\pi_{\vphi}$ is a boundary representation for $S$.
		\end{enumerate}
	\end{proposition}
	
	\begin{proof}
		Since $\vphi$ is pure, $\pi_{\vphi}$ is an irreducible representation. Hence $\pi_{\vphi}(A)'' = B(H_{\vphi})$. Therefore, maximality of $\vphi$ in $\pd$ is equivalent to maximality in $\ps$, proving $(1)\iff(2)$.
		
		By Theorem~\ref{T:dilationUEP}, $\vphi$ is maximal in $\pd$ if and only if $\pi_{\vphi}$ has the unique extension property. Since $\pi_{\vphi}$ is irreducible, this precisely means that $\pi_{\vphi}$ is a boundary representation for $S$. Hence $(1)\iff(3)$.
	\end{proof}

	\section{A state-theoretic amendment of Arveson's Hyperrigidity Conjecture}\label{S:ISD}
	
	Let $S$ be an operator system in the C*-algebra $A$ such that $\C(S)=A$. Then the state space $\mE(A)$ equipped with the weak-$*$ topology, is compact metrizable. For a Borel probability measure $\mu$ on $\mE(A)$, the \emph{barycenter} $\mb(\mu)\in \mE(A)$ is defined as
	\[
	\mb(\mu)(a)
	=
	\int_{\mE(A)} \alpha(a)\,d\mu(\alpha),
	\qquad a\in A.
	\]
	Let $\bar{\mu}$ denote the completion of $\mu$. Next we introduce a new relation defined on $\mE(A)$. 
	
	\begin{definition}[Integral subdivision relation]
		Let $\vphi, \psi \in \mE(A)$. We define the \emph{integral subdivision relation} $\isd$ on $\mE(A)$ by declaring that $\varphi\isd\psi$ if for every Borel measure $\mu$ with barycenter $\vphi$, there exists a $\bar{\mu}$-measurable field of states $\alpha \to \psi_{\alpha} \in \mE(A)$ such that,
		\begin{enumerate}
			\item $\alpha \pd \psi_{\alpha}$ for $\bar{\mu}$-almost every $\alpha \in \mE(A)$,
			\item for all $a \in A$, 
			\[
			\psi(a)= \int_{\mE(A)} \psi_{\alpha}(a)d\bar{\mu}(\alpha).
			\]
		\end{enumerate}
		Here the $\bar{\mu}$-measurability of the field means that for every $a \in A$, the scalar valued function $\alpha \mapsto \psi_{\alpha}(a)$ is $\bar{\mu}$-measurable. Finally, $\vphi \in \mE(A)$ is said to be \emph{maximal} in $\isd$ if
		\[
		\vphi \isd \psi \implies \psi = \vphi.
		\]
	\end{definition}
	The definition of the integral subdivision relation is motivated by a classical characterization of the Choquet order on the space of regular Borel probability measures on a compact metrizable convex set; see \cite[Chapter I.3]{alfsen2012compact}. This characterization was subsequently extended to the non-metrizable setting by Davidson and Kennedy in \cite{hyperrigidityforfunctions}. 
	
	Let $\vphi, \psi \in \mE(A)$ and $\vphi \isd \psi$. Let $\mu= \delta_{\vphi}$ be the Dirac measure at $\vphi$. Then $\mathfrak b(\mu)= \vphi$, and the definition of $\isd$ immediately implies that $\vphi \pd \psi$. Thus $\vphi \isd \psi$ implies $\vphi \pd \psi$.

	Next we prove two technical lemmas, before stating the main results of this section. 
	
	\begin{lemma}\label{L:finitebaricentriclifting}
		Let $\vphi\in \mE(A)$ be a state with the GNS representation $(\pi_{\vphi},H_{\vphi},\xi_{\vphi})$. Suppose that $\displaystyle \vphi=\sum_{i=1}^n \lambda_i\vphi_i$ is a finite convex subdivision of $\vphi$, where $\lambda_i>0$,
		$\displaystyle \sum_{i=1}^n\lambda_i=1$, and $\vphi_i\in \mE(A)$. Then for $1\leq i\leq n$,  there exist normal states $\gamma_i\in \mE(\pi_{\vphi}(A)'')$ such that $\gamma_i(\pi_{\vphi}(a))=\vphi_i(a)$ for all $a \in A$
		and
		\[
		\sum_{i=1}^n\lambda_i\gamma_i(T)
		=
		\langle T\xi_{\vphi},\xi_{\vphi}\rangle,
		\qquad T\in \pi_{\vphi}(A)''.
		\]
	\end{lemma}
	
	\begin{proof}
		Let $x \in A^{**}$ such that $\vphi(x^*x)=0$. Then $\lambda_i\vphi_i(x^*x)=0$ and hence $\vphi_i(x^*x)=0$ for all $1 \leq i \leq n$. This implies that $\vphi_i$ is absolutely continuous with respect to $\vphi$ (see \cite[Section 2.2]{clouatreTimko} for the definition of absolute continuity of states) for each $i$. By the absolute-continuity decomposition for states
		\cite[Proposition~4.2]{boundaryprojection}, for each $i$ there exist an orthonormal set $\{e_{i,k}:k\geq 1\}\subseteq H_{\vphi}$ and positive numbers
		$\{r_{i,k}:k\geq 1\}$ with $\displaystyle \sum_{k=1}^\infty r_{i,k}=1$
		such that
		\[
		\vphi_i(a)
		=
		\sum_{k=1}^\infty r_{i,k}
		\langle \pi_{\vphi}(a)e_{i,k},e_{i,k}\rangle,
		\qquad a\in A.
		\]
		Define $\gamma_i(T)=\displaystyle \sum_{k=1}^\infty r_{i,k} \langle Te_{i,k},e_{i,k}\rangle$ for all $T\in \pi_{\vphi}(A)''$. Then $\gamma_i$ is a normal state on $\pi_{\vphi}(A)''$, being a norm-convergent convex combination of normal vector states. Moreover,
		\[
		\gamma_i(\pi_{\vphi}(a))
		=
		\sum_{k=1}^\infty r_{i,k}
		\langle \pi_{\vphi}(a)e_{i,k},e_{i,k}\rangle
		=
		\vphi_i(a),
		\qquad a\in A.
		\]
		
		It remains to verify the barycenter identity on $\pi_{\vphi}(A)''$. For $a\in A$,
		\[
		\sum_{i=1}^n\lambda_i\gamma_i(\pi_{\vphi}(a))
		=
		\sum_{i=1}^n\lambda_i\vphi_i(a)
		=
		\vphi(a)
		=
		\langle \pi_{\vphi}(a)\xi_{\vphi},\xi_{\vphi}\rangle.
		\]
		Thus the normal functionals
		\[
		\sum_{i=1}^n\lambda_i\gamma_i
		\quad\text{and}\quad
		T\mapsto \langle T\xi_{\vphi},\xi_{\vphi}\rangle
		\]
		agree on $\pi_{\vphi}(A)$. Since $\pi_{\vphi}(A)$ is weak-$*$ dense in
		$\pi_{\vphi}(A)''$, and both functionals are normal, they agree on all of $\pi_{\vphi}(A)''$.
		Hence
		\[
		\sum_{i=1}^n\lambda_i\gamma_i(T)
		=
		\langle T\xi_{\vphi},\xi_{\vphi}\rangle,
		\qquad T\in \pi_{\vphi}(A)''.
		\]
		This proves the lemma.
	\end{proof}

	\begin{lemma} \label{L:lifting}
		Let $\vphi\in \mE(A)$, and let $(\pi_{\vphi},H_{\vphi},\xi_{\vphi})$ be the GNS representation of $\vphi$. Let $\mu$ be a Borel probability measure on $\mE(A)$ such that $ \mb(\mu)=\vphi$. Then there exists a family of states $\{\gamma_\alpha\}_{\alpha\in \mE(A)}\subseteq \mE(\pi_{\vphi}(A)'')$ such that
		\begin{enumerate}
			\item\label{i:1} for every $T\in \pi_{\vphi}(A)''$, the  function $\alpha\longmapsto \gamma_\alpha(T)$ is $\bar\mu$-measurable;
			
			\item\label{i:2} there exists a $\bar\mu$-measurable set $X\subseteq \mE(A)$ with $\bar\mu(X)=1$ such that
			\[
			\gamma_\alpha(\pi_{\vphi}(a))=\alpha(a),
			\qquad a\in A,
			\]
			for every $\alpha\in X$;
			
			\item\label{i:3} for every $T\in \pi_{\vphi}(A)''$,
			\[
			\int_{\mE(A)}\gamma_\alpha(T)\,d\bar\mu(\alpha)
			=
			\langle T\xi_{\vphi},\xi_{\vphi}\rangle.
			\]
		\end{enumerate}
	\end{lemma}
	
	\begin{proof}
		Let $(a_m)_{m=1}^\infty$ be a fixed countable dense subset of the unit ball of $A$. Then the state space $\mE(A)$, equipped with the weak-$*$ topology is compact metrizable with respect to the metric 
		\[
		d(\alpha, \beta)= \sum_{m=1}^{\infty} \frac{1}{2^m}| \alpha(a_m)-\beta(a_m)|.
		\]
		For each $n \in \mathbb{N}$, let $\mathcal P_n=\{B_{n,1},\dots,B_{n,k_n}\}$ be a finite disjoint Borel partition of $\mE(A)$ such that  $\operatorname{diam}(B_{n,j})<\frac1n$ in the mentioned metric. For each $j$ with $t_{n,j}:=\mu(B_{n,j})>0$, define a state $\alpha_{n,j}\in \mE(A)$ by
		\[
		\alpha_{n,j}(a)
		=
		\frac{1}{t_{n,j}}
		\int_{B_{n,j}}\alpha(a)\,d\mu(\alpha),
		\qquad a\in A.
		\]
		When $\mu(B_{n,j})=0$, choose arbitrary states $\alpha_{n,j}\in \mE(A)$.  Then for every $a\in A$ we have,
		\[
		\sum_{j=1}^{k_n}t_{n,j}\alpha_{n,j}(a)
		=
		\sum_{j=1}^{k_n}\int_{B_{n,j}}\alpha(a)\,d\mu(\alpha)  
		=
		\int_{\mE(A)}\alpha(a)\,d\mu(\alpha) 
		=
		\vphi(a).
		\]
		Hence we have a finite convex subdivision $\displaystyle \vphi=\sum_{j=1}^{k_n}t_{n,j}\alpha_{n,j}$. By Lemma \ref{L:finitebaricentriclifting}, for each $n$ and each $1 \leq j \leq k_n$ there exists a normal state $\gamma_{n,j}\in \mE(\pi_{\vphi}(A)'')$ such that $\gamma_{n,j}(\pi_{\vphi}(a))=\alpha_{n,j}(a)$ for all $a \in A$ and
		\[
		\sum_{j=1}^{k_n}t_{n,j}\gamma_{n,j}(T)
		=
		\langle T\xi_{\vphi},\xi_{\vphi}\rangle,
		\qquad  T\in \pi_{\vphi}(A)''.
		\]
		Let $\Gamma_n: \pi_{\vphi}(A)''\to L^\infty(\mE(A),\bar\mu)$ be the map defined as 
		
		\[
		\Gamma_n(T) =\sum_{j=1}^{k_n}\gamma_{n,j}(T)\mathbf 1_{B_{n,j}}, \qquad T \in \pi_{\vphi}(A)''
		\]
		where $\mathbf{1}_{B_{n,j}}$ is the characteristic function of $B_{n,j}$. The map $\Gamma_n$ is unital and positive, because each $\gamma_{n,j}$ is a state. Since the codomain algebra $L^\infty(\mE(A),\bar\mu)$ is commutative, $\Gamma_n$ is completely positive. Thus $\Gamma_n:\pi_{\vphi}(A)''\to L^\infty(\mE(A),\bar\mu)$ is UCP. Moreover, for every $T\in \pi_{\vphi}(A)''$,
		\[
		\begin{aligned}
			\int_{\mE(A)}\Gamma_n(T)\,d\bar\mu
			=
			\sum_{j=1}^{k_n}t_{n,j}\gamma_{n,j}(T) 
			=
			\langle T\xi_{\vphi},\xi_{\vphi}\rangle.
		\end{aligned}
		\]

		Now fix $a\in A$, and let $\widehat{a}\colon \mE(A)\to\mathbb{C}$ be the
		evaluation function defined by $\widehat{a}(\alpha)=\alpha(a)$ for all $\alpha \in \mE(A)$. Since $\widehat{a}$ is continuous on the compact metric space $\mE(A)$,
		it is uniformly continuous. Hence, for every $\varepsilon>0$, there
		exists $N_\varepsilon\in\mathbb{N}$ such that, for every $n\geq N_\varepsilon$ and every
		$\alpha,\beta\in \mE(A)$ satisfying $d(\alpha,\beta)<1/n$, we have
		\[
		\left|\widehat{a}(\alpha)-\widehat{a}(\beta)\right|<\varepsilon.
		\]
		
		Fix $n\geq N_\varepsilon$, and let $1\leq j\leq k_n$ be such that $t_{n,j}>0$.
		For every $\alpha\in B_{n,j}$, we have
		\[
		\begin{aligned}
			\left|\Gamma_n(\pi_\varphi(a))(\alpha)-\widehat{a}(\alpha)\right|
			&=
			\left|
			\frac{1}{t_{n,j}}
			\int_{B_{n,j}}\widehat{a}(\beta)\,d\mu(\beta)
			-\widehat{a}(\alpha)
			\right| \\
			&\leq
			\frac{1}{t_{n,j}}
			\int_{B_{n,j}}
			\left|\widehat{a}(\beta)-\widehat{a}(\alpha)\right|
			\,d\mu(\beta).
		\end{aligned}
		\]
		Since $\operatorname{diam}(B_{n,j})<1/n$, we have
		$d(\alpha,\beta)<1/n$ for all $\alpha,\beta\in B_{n,j}$. Therefore, $\left|\widehat{a}(\beta)-\widehat{a}(\alpha)\right|<\varepsilon$ throughout $B_{n,j}$, and hence
		\[
		\left|\Gamma_n(\pi_\varphi(a))(\alpha)-\widehat{a}(\alpha)\right|
		<\varepsilon.
		\]
		
		The union of those sets $B_{n,j}$ for which $t_{n,j}=0$ is
		$\overline{\mu}$-null. Thus the preceding estimate holds for
		$\overline{\mu}$-almost every $\alpha\in \mE(A)$. It follows that
		\[
		\left\| \Gamma_n(\pi_\varphi(a))-\widehat{a} \right\|_{L^\infty(\mE(A),\overline{\mu})} <\varepsilon
		\]
		for every $n\geq N_\varepsilon$. Therefore, $\Gamma_n(\pi_\varphi(a))$ converges to $\widehat{a}$ in the uniform norm for all $a \in A$.

		The set of all UCP maps from $\pi_\varphi(A)''$ to $L^\infty(\mE(A),\overline{\mu})$ is compact in the point weak-$^*$ topology. Hence there exist a cofinal subnet $(\Gamma_\lambda)_{\lambda\in\Lambda}$ of $(\Gamma_n)_{n\in\mathbb{N}}$, and a UCP map $\Gamma\colon \pi_\varphi(A)'' \longrightarrow L^\infty(\mE(A),\overline{\mu})$ such that $\Gamma_\lambda\to\Gamma$ in the point weak-$^*$ topology. Note that $(\Gamma_{\lambda}(\pi_\varphi(a)))_{\lambda \in \Lambda}$ is a subnet of $(\Gamma_n(\pi_\varphi(a)))_{n \in \mathbb N}$. Hence $\Gamma_{\lambda}(\pi_\varphi(a))$ converges to $\widehat{a}$ in norm and hence in the weak-$*$ topology. Therefor, by uniqueness of limits we have,
		\[
		\Gamma(\pi_\varphi(a))=\widehat{a}, \qquad a\in A.
		\]

		Moreover, integration against the constant function
		$1\in L^1(\mE(A),\overline{\mu})$ is weak-$*$ continuous 
		$L^\infty(\mE(A),\overline{\mu})$. Therefore, for every
		$T\in\pi_\varphi(A)''$,
		\[
		\begin{aligned}
			\int_{\mE(A)}\Gamma(T)\,d\overline{\mu} = \lim_{\lambda} \int_{\mE(A)}\Gamma_\lambda(T)\,d\overline{\mu}
			=
			\langle T\xi_\varphi,\xi_\varphi\rangle.
		\end{aligned}
		\]

		Let $\mathscr L^\infty(\mE(A),\bar\mu)$ denotes the bounded $\bar\mu$-measurable functions on $\mE(A)$. Then, as discussed in Subsection \ref{SS:Lifting}, there exists a unital $*$-homomorphism $\ell: L^{\infty}(\mE(A), \bar \mu) \to \mathscr L^\infty(\mE(A),\bar\mu)$ such that $\ell$ is unital, positive, and linear, and $\ell(f)$ is a representative of the class $f$.

		Let $\gamma_\alpha(T)= \ell(\Gamma(T))(\alpha)$ for $\alpha\in \mE(A)$ and $T\in \pi_{\vphi}(A)''$. For fixed $\alpha$, the map $T\longmapsto \gamma_\alpha(T)$ is linear, positive, and unital, because both $\Gamma$ and $\ell$ are linear, positive, and unital. Hence $\gamma_\alpha\in \mE(\pi_{\vphi}(A)'')$. For fixed $T\in \pi_{\vphi}(A)''$, the function $\alpha\longmapsto \gamma_\alpha(T)$ is $\bar\mu$-measurable, since it is the measurable function $\ell(\Gamma(T)) \in \mathscr L^\infty(\mE(A),\bar\mu)$. 
		
		Next recall that $\Gamma(\pi_{\vphi}(a))= \widehat{a}$ for all $a \in A$ and hence $\ell (\Gamma(\pi_{\vphi}(a)))= \widehat{a}$, $\bar{\mu}$-almost everywhere on $\mE(A)$. For all $m \in \mathbb{N}$, let $X_m = \{\alpha \in \mE(A): l (\Gamma(\pi_{\vphi}(a_m)))(\alpha)= \alpha(a_m)\}$ and let $X = \underset{m \in \mathbb{N}}{\bigcap} X_m$. Since $\bar{\mu}(X_m)=1$ for all $m$, $\bar{\mu}(X)= 1$ and for all $\alpha \in X$ and for all $m \in \mathbb{N}$,
		\[
		\gamma_{\alpha}(\pi_{\vphi}(a_m))=l (\Gamma(\pi_{\vphi}(a_m)))(\alpha)= \alpha(a_m).
		\]
		Since $(a_m)$ is norm dense on the unit ball, we have that $\gamma_{\alpha}(\pi_{\vphi}(a))= \alpha(a)$ for all $a \in A$ and for all $ \alpha \in X$.

		Finally, since $\ell(\Gamma(T))$ is a representative of the
		$L^\infty$-class $\Gamma(T)$, we have
		\[
		\begin{aligned}
			\int_{\mE(A)}\gamma_\alpha(T)\,d\bar\mu(\alpha)
			=
			\int_{\mE(A)}\ell(\Gamma(T))(\alpha)\,d\bar\mu(\alpha) 
			=
			\int_{\mE(A)}\Gamma(T)\,d\mu 
			&=
			\langle T\xi_{\vphi},\xi_{\vphi}\rangle,
		\end{aligned}
		\]
		for all $T \in \pi_{\vphi}(A)''$. This proves the lemma.
	\end{proof}

	Using the previous lemma, we now show that the strong dilation relation is stronger than the integral subdivision relation.

	\begin{theorem}\label{T:implication}
		Let $S$ be an operator system in a $\C$-algebra $A$ such that $\C(S)=A$. Let $\vphi,\psi\in \mE(A)$. If $\vphi\ps \psi$ then $\vphi\isd \psi$.
	\end{theorem}
	
	\begin{proof}
		Assume that $\vphi\ps \psi$ and let $(\pi_{\vphi},H_{\vphi},\xi_{\vphi})$ be the GNS representation of $\vphi$. By the definition of the strong dilation relation, there exists a unital completely positive map $\Phi:A\to \pi_{\vphi}(A)''$ such that $\Phi(s)=\pi_{\vphi}(s)$ for all $s \in S$ and $\psi(a)=\langle \Phi(a)\xi_{\vphi},\xi_{\vphi}\rangle$ for all $a \in A$. 
		
		Let $\mu$ be a regular Borel probability measure on $\mE(A)$ such that $\mb(\mu)=\vphi$. By Lemma \ref{L:lifting}, there exists a family of states $\{\gamma_\alpha\}_{\alpha\in \mE(A)}\subseteq \mE(\pi_{\vphi}(A)'')$ along with $X \subset \mE(A)$ with $\bar{\mu}(X)=1$ satisfying conditions (\ref{i:1}), (\ref{i:2}) and (\ref{i:3}) of Lemma \ref{L:lifting}.  For each $\alpha\in \mE(A)$, define a state $\psi_\alpha\in \mE(A)$ by $\psi_\alpha(a)=\gamma_\alpha(\Phi(a))$. 
		
		Since $\Phi(a) \in \pi_{\vphi}(A)''$, (\ref{i:1}) of Lemma \ref{L:lifting} says that the scalar functions $\alpha \mapsto \gamma_{\alpha}(\Phi(a))$ is $\bar{\mu}$-measurable for all $a \in A$. Hence $\alpha\mapsto \psi_\alpha$ is $\bar{\mu}$ measurable. Moreover by (\ref{i:3}), for every $a\in A$,
		\[
		\begin{aligned}
			\int_{\mE(A)}\psi_\alpha(a)\,d\bar\mu(\alpha)
			=
			\int_{\mE(A)}\gamma_\alpha(\Phi(a))\,d\bar\mu(\alpha) 
			=
			\langle \Phi(a)\xi_{\vphi},\xi_{\vphi}\rangle 
			=
			\psi(a).
		\end{aligned}
		\]

		Let $\alpha\in X$ and let $(\rho_\alpha,K_\alpha,\eta_\alpha)$ be the GNS representation of the state $\gamma_\alpha$ on the C*-algebra $\pi_{\vphi}(A)''$. Let $\sigma_\alpha:A\to B(K_\alpha)$ be defined as $\sigma_\alpha = \rho_{\alpha} \circ \pi_{\vphi}$. Then, for every $a\in A$,
		\[
		\begin{aligned}
			\langle \sigma_\alpha(a)\eta_\alpha,\eta_\alpha\rangle
			=
			\langle \rho_\alpha(\pi_{\vphi}(a))\eta_\alpha,\eta_\alpha\rangle 
			=
			\gamma_\alpha(\pi_{\vphi}(a)) 
			&=
			\alpha(a),
		\end{aligned}
		\]
		because $\alpha\in X$. Hence $(\sigma_\alpha,K_\alpha,\eta_\alpha)$ is a
		representation of the state $\alpha$. Let $\Theta_\alpha:A\to B(K_\alpha)$ be the UCP map defined as $\Theta_{\alpha}= \rho_\alpha \circ \Phi$. Then $\Theta_{\alpha}(s)= \sigma_{\alpha}(s)$ for all $s \in S$ and,
		\[
		\begin{aligned}
			\langle \Theta_\alpha(a)\eta_\alpha,\eta_\alpha\rangle
			=
			\langle \rho_\alpha(\Phi(a))\eta_\alpha,\eta_\alpha\rangle 
			=
			\gamma_\alpha(\Phi(a)) 
			&=
			\psi_\alpha(a).
		\end{aligned}
		\]
		for all $a \in A$. Then by Proposition \ref{P:equivalence} $\alpha\pd\psi_\alpha$ . Since this holds for every $\alpha\in X$, and since $\bar\mu(X)=1$, we have $\alpha\pd\psi_\alpha$ for $\bar\mu$--almost every $\alpha\in \mE(A)$. 
		Thus, for the arbitrary representing measure $\mu$ of $\vphi$, we have constructed a $\bar\mu$-measurable field of states $\alpha\longmapsto \psi_\alpha \in \mE(A)$ such that
		\[
		\alpha\pd\psi_\alpha
		\quad
		\bar\mu\text{-almost everywhere},
		\]
		and
		\[
		\psi(a)
		=
		\int_{\mE(A)}\psi_\alpha(a)\,d\bar\mu(\alpha),
		\qquad a\in A.
		\]
		This implies that $\vphi \isd \psi$. This proves the theorem.
	\end{proof}
	
	Next we state a straightforward consequence of Theorem \ref{T:implication}.
	\begin{corollary}\label{C:maximalitycompare}
		Let $S$ be an operator system contained in a $\C$-algebra $A$ such that $\C(S)=A$. Let $\vphi \in \mE(A)$ such that $\vphi$ is maximal in $\isd$. Then $\vphi$ is maximal in $\ps$. 
	\end{corollary} 
	\begin{proof}
		Let $\psi \in \mE(A)$ such that $\vphi \ps \psi$.  Then Theorem \ref{T:implication} implies that $\vphi \isd \psi$. Maximality of $\vphi$ in $\isd$ implies that $\psi=\vphi $. Hence $\vphi$ is maximal in $\ps$. 
	\end{proof}

	Next we show that the maximality of the pure states in the dilation order implies the maximality of the entire state space in the integral subdivision relation.

	\begin{proposition}\label{P:maximalitypassage}
		Let $S$ be an operator system contained in a $\C$-algebra $A$ such that $\C(S)=A$. If every pure state of $A$ is maximal in the dilation order $\pd$, then every state of $A$ is maximal in the integral subdivision relation $\isd$.
	\end{proposition}
	
	\begin{proof}
		Assume that every pure state of $A$ is maximal in $\pd$. Let $\vphi\in \mE(A)$. Let $\psi \in \mE(A)$ such that $\vphi \isd \psi$. Since $A$ is separable, the state space $\mE(A)$ is compact metrizable in the
		weak-$*$ topology. There exists a Borel probability measure $\mu$ on $\mE(A)$ such that $\mb(\mu)=\vphi$, and $\mu(\partial_e \mE(A))=1$ \cite[Theorem 4.2]{bishop_measure_extreme}. Here $\partial_e \mE(A)$ denotes the set of extreme points of $\mE(A)$, which is precisely the pure state space of $A$. $\mE(A)$ is metrizable as $A$ is separable \cite[Theorem 3.16]{Rudin_Functional} and hence $\partial_e\mE(A)$ is measurable \cite[Corollary I.4.4]{alfsen2012compact}.  Thus, for every $a\in A$, 
		\[
		\vphi(a)
		=
		\int_{\mE(A)} \alpha(a)\,d\mu(\alpha).
		\]
		Hence there exists a $\bar\mu$-measurable field of states $\alpha\longmapsto \psi_\alpha\in \mE(A)$ such that $\alpha \pd \psi_\alpha$ for $\bar\mu$-almost every $\alpha\in \mE(A)$ and
		\[
		\psi(a)
		=
		\int_{\mE(A)} \psi_\alpha(a)\,d\bar\mu(\alpha),
		\qquad a\in A.
		\]

		Let $X\subseteq \mE(A)$ be a $\bar\mu$-measurable set such that $\bar\mu(X)=1$ and $\alpha \pd\psi_\alpha$ for all $ \alpha \in X$. Since $\mu(\partial_e \mE(A))=1$, we also have that $\bar\mu\bigl(X\cap \partial_e \mE(A)\bigr)=1$. Let $\alpha\in X\cap \partial_e \mE(A)$ be a fixed element. Then $\alpha$ is a pure state and $\alpha\pd \psi_\alpha$. By hypothesis, $\psi_\alpha=\alpha$. Thus $\psi_\alpha=\alpha$ for $\bar{\mu}$-almost every $\alpha \in \mE(A)$. Therefore, for every $a\in A$,
		\[
		\begin{aligned}
			\psi(a)
			=
			\int_{\mE(A)} \psi_\alpha(a)\,d\bar\mu(\alpha) 
			=
			\int_{\mE(A)} \alpha(a)\,d\bar\mu(\alpha) 
			=
			\int_{\mE(A)} \alpha(a)\,d\mu(\alpha) 
			&=
			\vphi(a).
		\end{aligned}
		\]
		Hence $\psi=\vphi$. Therefore $\vphi$ is maximal in the integral subdivision dilation relation. every state of $A$ is maximal in
		$\isd$.
	\end{proof}
	
	Next we recover that the Clou\^atre-Thompson rigidity phenomenon as a consequence of the previous theorem. Moreover it is worthwhile to note that this order-theoretic approach is independent of the techniques used by Clou\^atre and Thompson in \cite{clouatre2024rigidity}

	\begin{corollary}\label{C:CTamended}
		Let $S$ be an operator system in a  $\C$-algebra $A$ such that $\C(S)=A$. Then the following are equivalent
		\begin{enumerate}
			\item Every irreducible representation is boundary representation. 
			\item Every representation has the unique tight extension property. 
		\end{enumerate}
	\end{corollary}

	\begin{proof}
		Let $\pi: A \to B(H)$ be an irreducible representation. Since $\pi(A)''=B(H)$, the unique extension property and the unique tight extension property coincide for $\pi$, so that $(2) \implies (1)$.

		Next, we assume that every irreducible representation is a boundary representation. Equivalently, this means that every pure state is maximal in $\pd$. Then by Proposition \ref{P:maximalitypassage} every state of $A$ is maximal in $\isd$. Hence every state is maximal in $\ps$ due to Corollary \ref{C:maximalitycompare}. Thus by Theorem \ref{T:UTEPmaximality}, every cyclic representation has the unique tight extension property. Therefore every representation has the unique tight extension property with respect to $S$ due to Proposition \ref{P:UTEPdirectsum} and hence $(1)\implies (2)$.
	\end{proof}

	\begin{remark}
		At the level of the state space, Arveson's hyperrigidity conjecture predicted the following implication: if every pure state is maximal for the dilation order, then every state is maximal for the dilation order. This implication is now known to be false. In Proposition~\ref{P:maximalitypassage}, we established the following weaker statement: if every pure state is maximal for the dilation order, then every state is maximal for the integral subdivision relation. Thus, Proposition~\ref{P:maximalitypassage} may be viewed as an amendment of Arveson's hyperrigidity conjecture at the state space level.
	\end{remark}

	\section{Comparison of the strong dilation and the integral subdivision relations}\label{S:comparison}

	\subsection{Equivalence of the relations in the multiplicity-free case}\label{SS:multiplicityfree}
	We have shown that, for states $\vphi,\psi$ of $A$, the relation $\vphi\ps\psi$ implies $\vphi\isd\psi$. The main goal of this subsection is to show that the converse implication also holds in certain cases. Next we prove a technical lemma.

	\begin{lemma}\label{L:commutingprojections}
		Let  $S$ be an operator system in a $\C$-algebra $A$ such that $\C(S)=A$. Let $\vphi,\psi\in \mE(A)$, and let $(\pi_\vphi,H_\vphi,\xi_\vphi)$ be the GNS representation of $\vphi$. Suppose that $\vphi \isd \psi $. Then, for every finite set $\mathcal F$ of pairwise commuting projections in $\pi_\vphi(A)'$, there exists a UCP map $\Phi_{\mathcal F}:A\to B(H_\vphi)$ such that
		\begin{enumerate}
			\item $\Phi_{\mathcal F}(s)=\pi_\vphi(s)$ for every $s\in S$;
			\item $\psi(a)=\langle \Phi_{\mathcal F}(a)\xi_\vphi,\xi_\vphi\rangle$ for every
			$a\in A$;
			\item $\Phi_{\mathcal F}(a)P=P\Phi_{\mathcal F}(a)$ for every
			$a\in A$ and every $P\in\mathcal F$.
		\end{enumerate}
	\end{lemma}
	
	\begin{proof}
		Write $\mathcal F=\{P_1,\ldots,P_n\}$. Let $\mathcal G=\{Q_1,\ldots,Q_m\}$ be the collection of nonzero projections of the form $Q=X_1\cdots X_n$, where $X_k\in\{P_k,I-P_k\}$. Then the $Q_j$'s are pairwise orthogonal projections in $\pi_\vphi(A)'$ and
		\[
		\sum_{j=1}^m Q_j=I.
		\]
		Put $K_j=Q_jH_\vphi$, which reduces $\pi_\vphi(A)$.
		Let $\pi_j(a)=\pi_\vphi(a)|_{K_j}$. Note that $Q_j\xi_\vphi\neq 0$, since $\xi_\vphi$ is cyclic for
		$\pi_\vphi(A)$ and $Q_j\neq 0$. Set
		\[
		t_j=\|Q_j\xi_\vphi\|^2,
		\qquad
		\xi_j=t_j^{-1/2}Q_j\xi_\vphi,
		\]
		and define $\vphi_j(a)=\langle \pi_j(a)\xi_j,\xi_j\rangle$ for all $a \in A$. Then $(\pi_j,K_j,\xi_j)$ is the GNS representation of $\vphi_j$, and $\displaystyle \vphi=\sum_{j=1}^m t_j\vphi_j$. Let $\displaystyle \mu=\sum_{j=1}^m t_j\delta_{\vphi_j}$. Then $\mu$ is a Borel probability measure on $\mE(A)$ such that  $\mb(\mu)=\vphi$. Since $\vphi\isd\psi$, applying the definition to $\mu$ gives states $\psi_j\in \mE(A)$ such that $\phi_j\pd\psi_j$ for all $1\leq j\leq m$ and
		\[
		\psi=\sum_{j=1}^m t_j\psi_j.
		\]
		For each $j$, choose a UCP map $\Phi_j:A\to B(K_j)$ witnessing $\phi_j \pd \psi_j$ with respect to the GNS representation
		$(\pi_j,K_j,\xi_j)$; that is, $\Phi_j(s)=\pi_j(s)$ for all $s \in S$ and $\psi_j(a)=\langle \Phi_j(a)\xi_j,\xi_j\rangle$ for all $a \in A$. Let 
		\[
		\Phi_{\mathcal F}
		=
		\bigoplus_{j=1}^m \Phi_j:
		A\to B\left(\bigoplus_{j=1}^mK_j\right)=B(H_\vphi).
		\]
		Then $\Phi_{\mathcal F}$ is UCP, and $\Phi_{\mathcal F}(s)=\pi_\vphi(s)$   for all $s\in S$, Moreover, for $a\in A$,
		\[
		\begin{aligned}
			\langle \Phi_{\mathcal F}(a)\xi_\vphi,\xi_\vphi\rangle
			=
			\sum_{j=1}^m
			\langle \Phi_j(a)Q_j\xi_\vphi,Q_j\xi_\vphi\rangle 
			&=
			\sum_{j=1}^m
			t_j\langle \Phi_j(a)\xi_j,\xi_j\rangle \\
			&=
			\sum_{j=1}^m t_j\psi_j(a)
			=
			\psi(a).
		\end{aligned}
		\]
		Finally, $\Phi_{\mathcal F}(a)$ is block diagonal with respect to
		$H_\vphi=\bigoplus_jK_j$, and hence commutes with each $Q_j$. Since each
		$P_k\in\mathcal F$ is a sum of some of the $Q_j$'s, $\Phi_{\mathcal F}(a)$
		commutes with every $P\in\mathcal F$.
	\end{proof}
	
	Using Lemma~\ref{L:commutingprojections}, we next show that if $\varphi,\psi\in\mE(A)$ and the GNS representation of $\varphi$ is multiplicity-free, then
	\[
	\varphi\ps\psi
	\quad\Longleftrightarrow\quad
	\varphi\isd\psi.
	\]
	Recall that a representation $\pi:A\to B(H)$ is said to be \emph{multiplicity-free} if its commutant $\pi(A)'$ is commutative \cite[Definition 6.27]{TakesakiI}.

	\begin{theorem} \label{T:commutativeequivalence}
		Let $S$ be an operator system in a $\C$-algebra $A$ such that $\C(S)=A$. Let $\vphi,\psi\in \mE(A)$ and the GNS representation of $\vphi$ is multiplicity-free. Then $\vphi \ps \psi$ if and only if $\vphi \isd \psi$. 
	\end{theorem}
	
	\begin{proof}
		Let $(\pi_{\vphi},H_\vphi,\xi_\vphi)$ be the GNS representation of $\vphi$. Since $\pi_\vphi$ is multiplicity-free, we have that $\pi_\vphi(A)'$ is commutative. 
		
		First assume that $\vphi \isd \psi$. Let $\Lambda$ be the directed set of all finite sets of projections in $\pi_\vphi(A)'$, ordered by inclusion. Since $\pi_\vphi(A)'$ is commutative, every $\mathcal{F}\in\Lambda$ is a finite set of pairwise commuting projections. By Lemma~\ref{L:commutingprojections}, for every $\mathcal{F}\in\Lambda$ there exists a UCP map $\Phi_\mathcal{F}:A\to B(H_\vphi)$ such that
		\[
		\Phi_\mathcal{F}(s)=\pi_\vphi(s),
		\qquad s\in S,
		\]
		\[
		\psi(a)=\langle \Phi_\mathcal{F}(a)\xi_\vphi,\xi_\vphi\rangle,
		\qquad a\in A,
		\]
		and
		\[
		\Phi_\mathcal{F}(a)P=P\Phi_\mathcal{F}(a),
		\qquad a\in A,\ P\in \mathcal{F}.
		\]
		The set of UCP maps from $A$ to $B(H_\vphi)$ is compact in the point weak-operator topology. Passing to a cofinal subnet, we may assume that $\Phi_\mathcal{F}\to \Phi$ in the point weak-operator topology, for some UCP map $\Phi:A\to B(H_\vphi)$. It follows immediately that $\Phi(s)=\pi_\vphi(s)$ for all $s \in S$ and $\psi(a)=\langle \Phi(a)\xi_\vphi,\xi_\vphi\rangle$ for all $a \in A$.  \par 
		It remains to show that $\Phi(a)\in\pi_\vphi(A)''$ for every $a\in A$. Fix a
		projection $P\in\pi_\vphi(A)'$, and set $ \Lambda_P= \{ \mathcal{F}\in\Lambda:P\in \mathcal{F}\}$. Then $\Lambda_P$ is cofinal in $\Lambda$. Hence the subnet indexed by
		$\Lambda_P$ also converges point weak-operator to $\Phi$. By construction, $\Phi_\mathcal{F}(a)P=P\Phi_\mathcal{F}(a)$ for every $a\in A$. Passing to the weak-operator limit gives
		\[
		\Phi(a)P=P\Phi(a),
		\qquad a\in A.
		\]
		Thus $\Phi(a)$ commutes with every projection in $\pi_\vphi(A)'$. Since a von Neumann algebra is generated by its projections, $\Phi(a)$ commutes with all of $\pi_\vphi(A)'$. Therefore
		\[
		\Phi(a)\in(\pi_\vphi(A)')'=\pi_\vphi(A)''.
		\]
		
		Thus $\Phi:A\to\pi_\vphi(A)''$ is a UCP map satisfying $\Phi(s)=\pi_\vphi(s)$ for all $s \in S$ and $\psi(a)=\langle \Phi(a)\xi_\vphi,\xi_\vphi\rangle$ for all $a \in A$. Hence $\vphi \ps \psi$. This proves that $\vphi \isd \psi \implies \vphi \ps \psi$ when $\pi_\vphi: A \to B(H_\vphi)$ is multiplicity-free. \par 
		The other implication does not require the GNS representation of $\vphi$ to be multiplicity-free and is due to Theorem \ref{T:implication}. This completes the proof.
	\end{proof}
	
	The next result is also an immediate consequence of Theorem \ref{T:commutativeequivalence}. 
	
	\begin{corollary}
		Let $S$ be an operator system in a $\C$-algebra $A$ such that $\C(S)=A$. Let $\vphi\in \mE(A)$ and assume that the GNS representation of $\vphi$ is multiplicity free. Then $\vphi$ is maximal in the strong dilation relation if and only if $\vphi$ is maximal in the integral subdivision relation. 
	\end{corollary}
	\begin{proof}
		Let $\vphi$ be maximal in $\isd$. Then by Corollary \ref{C:maximalitycompare}, $\vphi$ is maximal in $\ps$. Conversely, let $\vphi$ be maximal in $\ps$ and let $\psi \in \mE(A)$ such that $\vphi \isd \psi$. Since the GNS representation of $\vphi$ is multiplicity-free, Theorem \ref{T:commutativeequivalence} implies that $\vphi \ps \psi$. But maximality of $\vphi$ in $\ps$, implies that $\psi= \vphi$. Hence $\vphi$ is maximal in $\isd$.
	\end{proof}

	Since every cyclic representation of a commutative $\C$-algebra is multiplicity-free, the preceding theorem immediately yields the following corollary.
	
	\begin{corollary}
		Let $S$ be an operator system in a commutative $\C$-algebra $A$ such that $\C(S)=A$. Let $\vphi,\psi\in \mE(A)$.  Then the following statements are true.
		\begin{enumerate}
			\item $\vphi \ps \psi$ if and only if $\vphi \isd \psi$.
			\item $\vphi$ is maximal in $\ps$ if and only if $\vphi$ is maximal in $\isd$.
		\end{enumerate}
	\end{corollary}

	Motivated by the amended form of Arveson's hyperrigidity conjecture, we introduced the strong dilation relation $\ps$ and the integral subdivision relation $\isd$. Our results yield a state-space analogue of the Clou\^atre--Thompson rigidity phenomenon and, a priori, a stronger maximality statement: \emph{if every pure state is maximal in the dilation order, then every state is maximal for $\isd$}. To determine whether this genuinely strengthens the Clou\^atre-Thompson theorem, one must know whether $\ps$ and $\isd$ can differ, or  whether they can have different maximal states. We do not know whether this occurs. Since the two relations coincide in the commutative and multiplicity-free settings, any counterexample must lie beyond those classes. We therefore leave this as an open problem.
	
	\medskip
	
	\noindent\textbf{Question.}
	Can the strong dilation relation $\ps$ and the integral subdivision relation
	$\isd$ differ? More specifically, can they have different maximal states?
	
	\medskip

	\subsection{Equivalence of Choquet order, strong dilation relation and integral subdivision relation in the commutative case}\label{SS:Choquet}

	We conclude this section by relating the state-space relations considered above with abstract Choquet theory.  For a more systematic discussion on abstract Choquet theory we refer the reader to \cite{alfsen2012compact}.
	
	Let $X$ be a compact metrizable space, and let $S\subset C(X)$ be a function system that separates the points of $X$. Then by  Stone-Weierstrass theorem, $\C(S)=C(X)$. Let $\mE(S)$ denote the state space of $S$. We consider the evaluation map $\iota:X\to \mE(S)$ defined as $\iota(x)(s)=s(x)$ for all $s \in S$. For $s\in S_{\mathrm{sa}}$, define $\widehat{s}:\mE(S)\to \mathbb R$ as $\widehat{s}(\rho)= \rho(s)$ for all $\rho \in \mE(S)$. Then $\widehat{s}$ is continuous and affine, and $\widehat{s}\circ \iota=s$.
	
	Let $\mathcal C_S$ denote the uniformly closed max-stable cone in
	$C(X)_{\mathrm{sa}}$ generated by $S_{\mathrm{sa}}$. Thus $\mathcal C_S$ is the
	norm-closed cone generated by the positive linear span of functions of the form
	\[
	\max\{s_1,\dots,s_n\},
	\qquad s_1,\dots,s_n\in S_{\mathrm{sa}}.
	\]
	
	\begin{proposition}\label{P:max-stable-cone}
		With notation as above,
		\[
		\mathcal C_S= \overline{\{F\circ\iota:F\in C(\mE(S))\text{ is convex}\}}^{\|\,\|_{\infty}}.
		\]
	\end{proposition}
	
	\begin{proof}
		By the Kadison representation theorem for function systems, the map $s\mapsto \widehat{s}$ identifies $S_{\mathrm{sa}}$ with the real continuous affine functions on $\mE(S)$ \cite{Kadison}. Hence finite maxima of elements of $S_{\mathrm{sa}}$ are precisely pullbacks along $\iota$ of finite maxima of continuous affine functions on $\mE(S)$.
		
		Every finite maximum of continuous affine functions is continuous and convex.
		Thus
		\[
		\mathcal C_S \subseteq \overline{ \{F\circ\iota:F\in C(\mE(S))\text{ is convex}\}}^{\|\,\|_{\infty}}.
		\]
		Conversely, by the standard approximation theorem for compact convex sets,
		every continuous convex function on $\mE(S)$ is the uniform limit of finite maxima
		of continuous affine functions. Pulling back along $\iota$, it follows that $F\circ\iota\in\mathcal C_S$ whenever $F\in C(\mE(S))$ is convex. This proves the claim.
	\end{proof}
	For the rest of this section, we identify a Borel probability measure $\mu$ on
	$X$ with the state on $C(X)$ given by
	\[
	\mu(f)=\int_X f\,d\mu,
	\qquad f\in C(X).
	\]
	The pushforward of $\mu$ under $\iota: X \to \mE(S)$ denoted by $\iota_*\mu$, is defined as 
	\[
	\iota_*\mu(E)= \mu(\iota^{-1}(E)),
	\]
	for all Borel measurable subset of $\mE(S)$.
	
	Next we state the definition of abstract Choquet order on the state space of a commutative $\C$-algebra. For a detailed discussion on this topic we refer to \cite[I.5]{alfsen2012compact}
	
	\begin{definition}
		Let $\mu,\nu$ be Borel probability measures on $X$. We say that $\mu$ is
		dominated by $\nu$ in the \emph{abstract Choquet order} generated by $S$, and write $\mu\preceq_{\mathrm C}\nu$ if
		\[
		\mu(f)\leq \nu(f),
		\qquad \forall f\in \mathcal C_S.
		\]
	\end{definition}
	Note that if we start with a compact convex set $K$ and function system $S=A(K)$ of all continuous affine functions on $K$, then $\mathcal{C}_S$ is precisely the set of all continuous convex functions on $K$. Hence in this setup, $\preceq_{\mathrm{C}}$ is the classical Choquet order determined by convex functions \cite{alfsen2012compact}, \cite{phelps2002lectures}. Next we show that the relations we considered are equivalent to the abstract Choquet order.

	\begin{theorem}\label{T:choquet-isd-characterization}
		Let $X$ be a compact, metrizable space, and let $S\subset C(X)$ be a function system such that $S$ separates the points of $X$. Let $\mu,\nu$ be Borel probability measures on $X$. Then the following are equivalent:
		\begin{enumerate}
			\item $\mu\preceq_\mathrm{C}\nu$, that is,
			\[
			\int_X f\,d\mu\leq \int_X f\,d\nu, \qquad f\in \mathcal C_S.
			\]
			
			\item $\mu\ps\nu$.
			
			\item $\mu \isd \nu$.
			\item there exists a $\mu$-measurable field of probability
			measures $ x\mapsto \lambda_x\in \operatorname{Prob}(X)$ such that
			\[
			\lambda_x(s)=s(x), \qquad s \in S
			\]
			for $\mu$-almost every $x\in X$, and
			\[
			\int_X f\,d\nu
			=
			\int_X\lambda_x (f)d\mu(x),
			\qquad f\in C(X).
			\]
		\end{enumerate}
	\end{theorem}
	
	\begin{proof}
		Let $K=\mathcal E(S)$ and $Y=\iota(X)\subset K$. Since the elements of $S$ separates the points of $X$, the evaluation map $\iota:X\to K$ is injective, and hence $\iota$ is a homeomorphism from $X$ onto the compact subset $Y$.
		
		Assume that $\mu \preceq_\mathrm{C} \nu$.  By Proposition~\ref{P:max-stable-cone}, this implies that $\iota_*\mu\preceq_\mathrm{C} \iota_*\nu$ in the usual Choquet order on the compact convex set $K$. Indeed, for every
		continuous convex $F\in C(K)$,
		\[
		\int_K F\,d(\iota_*\mu)
		=
		\int_X F\circ\iota\,d\mu
		\leq
		\int_X F\circ\iota\,d\nu
		=
		\int_K F\,d(\iota_*\nu).
		\]
		Then by Theorem \ref{T:Choquet} there exists an $\iota_*\mu$-measurable field of representing measures $k \mapsto \sigma_k\in \operatorname{Prob}(K)$ such that
		\[
		\sigma_k(\mathfrak a)= \mathfrak a(k), \qquad \mathfrak a \in A(K)
		\]
		for $\iota_*\mu$-almost every $k \in K$, and  
		\[
		\iota_*\nu(g)
		=
		\int_K \sigma_k(g)\,d(\iota_*\mu)(k),
		\qquad g\in C(K).
		\]
		Since $\iota_*\nu$ is supported on $Y$, we have
		\[
		1
		=
		\iota_*\nu(Y)
		=
		\int_K \sigma_k(Y)\,d(\iota_*\mu)(k).
		\]
		Thus $\sigma_k(Y)=1$ for $\iota_*\mu$-almost every $k$. Let $Z = \{k \in K: \sigma_k(Y)=1 \}$. Then $\iota _*\mu(Z \cap Y)=1$ and let $X_0= \iota^{-1}(Z \cap Y)$. Then $\mu(X_0)=1$ and for all $x \in X_0$ define $\lambda_x=(\iota^{-1})_*\sigma_{\iota(x)}$.  On the remaining null set, define $\lambda_x$ arbitrarily. The field
		$x\mapsto \lambda_x$ is $\mu$-measurable: it is obtained by composing the
		measurable field $k\mapsto\sigma_k$ with the Borel map
		$x\mapsto\iota(x)$, and then applying the Borel pushforward map induced by the
		homeomorphism $\iota^{-1}:Y\to X$.
		
		For $s\in S$ and for all $x \in X_0$, since $\sigma_{\iota(x)}$ represents $\iota(x)$, we get
		\[
		\lambda_x(s)
		=
		\int_Y s\circ\iota^{-1}\,d\sigma_{\iota(x)}
		=
		\int_K \widehat{s}\,d\sigma_{\iota(x)}
		=
		\widehat{s}(\iota(x))
		=
		s(x).
		\]
		Finally, if $f\in C(X)$, choose $g\in C(K)$ such that $g|_Y=f\circ\iota^{-1}$. Note that $\sigma_k(g)= \displaystyle \int_Ygd\sigma_k$ for all $k \in Z \cap Y$.
		
		Then, since both $\sigma_{\iota(x)}$ and $\iota_*\nu$ are supported on $Y$,

		\begin{align*}
			\int_X \lambda_x(f)\,d\mu(x) = \int_{X_0}\lambda_x(f)\,d\mu(x) &= \int_{X_0}(\iota^{-1})_*\sigma_{\iota(x)}(f)d\mu(x)\\
			&= \int_{X_0} \sigma_{\iota(x)}(f\circ \iota^{-1})d\mu(x)\\
			&= \int_{X_0}\sigma_{\iota(x)}(g)d\mu\\
			&= \int_{Z\cap Y}\sigma_k(g)d(\iota_*\mu)(k)\\
			&=\int_K \sigma_k(g)d(\iota_*\mu)(k)\\
			&= \iota_*\nu(g)=\nu(f).
		\end{align*}
		This proves that (1) implies (4). 
		
		Conversely, assume (4). Let $F\in C(K)$ be convex. For $\mu$-almost every $x$,
		the measure $\iota_*\lambda_x$ represents $\iota(x)$, since for every $s\in S$,
		\[
		\int_K \widehat{s}\,d(\iota_*\lambda_x)
		=
		\lambda_x(s)
		=
		s(x)
		=
		\widehat{s}(\iota(x)).
		\]
		Hence Jensen's inequality gives
		\[
		F(\iota(x))
		\leq
		\int_K F\,d(\iota_*\lambda_x)
		=
		\lambda_x(F\circ\iota).
		\]
		Integrating over $X$, we obtain
		\[
		\int_X F\circ\iota\,d\mu
		\leq
		\int_X \lambda_x(F\circ\iota)\,d\mu(x)
		=
		\nu(F\circ\iota).
		\]
		By Proposition~\ref{P:max-stable-cone}, this is precisely
		$\mu\preceq_{\mathrm C}\nu$. Therefore $(1)$ and $(4)$ are equivalent. \par 
		Next assume that (4) holds. Then define $\Phi(f)(x)=\lambda_x(f)$ for all $f \in C(X)$. Then $\Phi:C(X)\to L^\infty(X,\mu)$ is unital and positive, satisfies
		$\Phi(s)=s$ for $s\in S$, and
		\[
		\nu(f)=\int_X\Phi(f)\,d\mu= \langle \Phi(f)\mathbf 1_X,\mathbf 1_X \rangle ,
		\qquad f\in C(X).
		\]
		Since the GNS representation of $\mu$ is the multiplication representation of
		$C(X)$ on $L^2(X,\mu)$, with bicommutant $L^\infty(X,\mu)$, this $\Phi$ is a strong dilation witness for $\mu\ps\nu$. This implies that (4) implies (2). \par
		Next assume that $(2)$ holds, that is let $\mu \ps\nu$. Then there exists a UCP map $\Phi: C(X) \to L^{\infty}(X,\mu)= \pi_\mu(C(X))''$ such that $\Phi(s)= \pi_\mu(s)$ for all $s \in S$ and $\nu(f)= \langle \Phi(f)\mathbf 1_X, \mathbf 1_X \rangle$ for all $f \in C(X)$. For all $x \in X$, let $\psi_x: C(X) \longrightarrow \mathbb C$ given by $\psi_x(f)= \ell (\Phi(f))(x)$, where $\ell\colon L^\infty(X,\mu) \longrightarrow \mathscr L^\infty(X,\mathcal A_\mu)$ is the lifting map mentioned in Subsection \ref{SS:Lifting}. Thus $\psi_x$ is a positive linear functional for $\mu$-a.e. $x \in X$. Thus by Riesz-Markov-Kakutani theorem, for $\mu$-a.e. $x \in X$, there exists Borel regular probability measures $\lambda_x$ on $X$ such that $\psi_x(f)= \displaystyle \int_X fd\lambda_x$ for all $f \in C(X)$. Thus $x \mapsto \lambda_x$ is a $\mu$-measurable field of probability measures such that
		\[
		\lambda_x(s)= \ell (\Phi(s))(x)= \ell (\pi_\mu(s))(x)= s(x)
		\]
		for $\mu$-almost every $x \in X$. Finally for all $f \in C(X)$, we have 
		\[
		\nu(f)= \int_X\Phi(f)d\mu= \int_X\ell(\Phi(f))(x)d\mu(x)= \int_X\lambda_x(f)d\mu(x). 
		\]
		This proves that $(2)$ implies $(4)$ and hence $(2)$ and $(4)$ are equivalent. \par 
		The equivalence of (2) and (3) is due to Theorem \ref{T:commutativeequivalence}. This completes the proof. 
	\end{proof}

	\section*{Acknowledgments}
	
	Part of this work was completed during the author's doctoral studies at the University of Manitoba, with financial support from the university. The remaining work was completed at the University of Haifa and was supported by a NSF--BSF research grant held by Adam Dor-On. 
	
	The author is grateful to Rapha\"el Clou\^atre and Adam Dor-On for their valuable remarks and suggestions.

	\bibliographystyle{plain}
	\bibliography{bibliography}
	
\end{document}